\documentclass{article}
\usepackage{graphicx} 
\usepackage{hyperref}
\usepackage{xcolor}
\usepackage{amsthm}
\usepackage[square,numbers,sort]{natbib}
\usepackage{float}
\usepackage{caption}
\usepackage{authblk}

\title{Small planar hypohamiltonian graphs}
\date{}

\author{Cheng-Chen Tsai\footnote{Email address: \href{mailto:u_e_smvdv@yahoo.com.tw}{u\_e\_smvdv@yahoo.com.tw}}}
\affil{\normalsize Institute of Information Science, Academia Sinica, Taiwan}

\hypersetup{
	colorlinks=true,
    citecolor=blue,
    linkcolor=blue,
    urlcolor=blue
}

\newtheoremstyle{plain}
{}
{}
{\slshape}
{}
{\bfseries}
{.}
{5pt}
{}

\newtheorem{theorem}{Theorem}[section]
\newtheorem{lemma}[theorem]{Lemma}
\newtheorem{corollary}[theorem]{Corollary}

\newcommand{\nextLine}{\hphantom{a}}
\newcommand{\mmod}[3]{#1\equiv#2\ (\mathrm{mod}\ #3)}

\begin{document}

\maketitle
\vspace{-38pt}
\section*{Abstract} A graph is hypohamiltonian if it is non-Hamiltonian, but the deletion of every single vertex gives a Hamiltonian graph. Until now, the smallest known planar
hypohamiltonian graph had 40 vertices, a result due to Jooyandeh, McKay, {\" O}sterg{\aa}rd, Pettersson, and Zamfirescu. That result is here improved upon by two planar hypohamiltonian graphs on 34 vertices. We exploited a special subgraph contained in two graphs of Jooyandeh et al., and modified it to construct the two 34-vertex graphs and six planar hypohamiltonian graphs on 37 vertices. Each of the 34-vertex graphs has 26 cubic vertices, improving upon the result of Jooyandeh et al. that planar hypohamiltonian graphs have 30 cubic vertices. We use the 34-vertex graphs to construct hypohamiltonian graphs of order 34 with crossing number 1, improving the best-known bound of 36 due to Wiener. Whether there exists a planar hypohamiltonian graph on 41 vertices was an open question. We settled this question by applying an operation introduced by Thomassen to the 37-vertex graphs to obtain several planar hypohamiltonian graphs on 41 vertices. The 25 planar hypohamiltonian graphs on 40 vertices of Jooyandeh et al. have no nontrivial automorphisms. The result is here improved upon by six planar hypohamiltonian graphs on 40 vertices with nontrivial automorphisms.

\section{Introduction}
A graph is \emph{hypohamiltonian} if it is non-Hamiltonian, but the deletion of every single vertex gives a Hamiltonian graph. A graph \(G=(V,E)\) is \emph{almost hypohamiltonian} if there exists a \(w\in V\) such that \(G-w\) is non-Hamiltonian, but for every vertex \(v\neq w\) in \(G\), the graph \(G-v\) is hamiltonian. We shall call \(w\) the
\emph{exceptional vertex} of \(G\).

In 1973, Chv{\'a}tal \cite{Chvátal_1973} raised the problem whether there exist planar hypohamiltonian graphs, and offered \$5 for its solution (Problem 19 of \cite{Chvátal_1972}). In 1974, Gr{\"u}nbaum \cite{GRUNBAUM197431} conjectured that there are no planar hypohamiltonian graphs. In 1976, Thomassen \cite{THOMASSEN1976377} gave infinitely many counterexamples, the smallest being of order 105. Various authors have improved this upper bound for the order of the smallest planar hypohamiltonian graph to order 57 (Hatzel \cite{HATZEL79} in 1979), 48 (C. T. Zamfirescu and T. Zamfirescu \cite{ZamfirescuZ07} in 2007), 42 (Araya and Wiener \cite{WienerA11} in 2011), and 40 (Jooyandeh et al. \cite{JooyandehMOPZ17} in 2017).

The lower bound for the order of the smallest planar hypohamiltonian graph was given by J. Goedgebeur and C. T. Zamfirescu \cite{GoedgebeurZ17}.
\begin{theorem}[{\cite[Theorem 3.2]{GoedgebeurZ17}}]\label{thm:gz3_2}
The smallest planar hypohamiltonian graph has at least \(23\) vertices.
\end{theorem}
\begin{theorem}[{\cite[Theorem 3.3]{GoedgebeurZ17}}]\label{thm:gz3_3} The smallest planar hypohamiltonian graph with girth at least \(4\) has at least \(27\) vertices.
\end{theorem}

J. Goedgebeur and C. T. Zamfirescu \cite{GoedgebeurZ17} defined \(h\) (\(h_g\)) to be the order of the smallest planar hypohamiltonian graph (of girth \(g\)). Combining \autoref{thm:gz3_2} and \autoref{thm:gz3_3} with the result of \cite{JooyandehMOPZ17}, we have \(23\leq h\leq 40\), \(27\leq h_4\leq 40\), and \(h_5=45\) \cite[Corollary 3.4]{GoedgebeurZ17}.

T. Zamfirescu \cite{ZAMFIRESCU1972116} defined \(\overline{C^i_k}\) and \(\overline{P^i_k}\) to be the smallest order for which there is a planar \(k\)-connected graph such that every set of \(i\) vertices is disjoint from some longest cycle and path, respectively. The best bounds known so far were \(\overline{C^1_3}\leq 40\), \(\overline{C^2_3}\leq 2625\), \(\overline{P^1_3}\leq 156\), \(\overline{P^2_3}\leq 10350\), which were found based on a planar hypohamiltonian graph on 40 vertices \cite{JooyandehMOPZ17}, and a cubic planar hypohamiltonian graph on 70 vertices \cite{AW11}.

C. T. Zamfirescu \cite{Zamfirescu15} defined \(\overline{\alpha}_0\) (\(\overline{\alpha}_1\)) to be the order of the smallest planar hypohamiltonian (smallest planar almost hypohamiltonian) graph. To emphasize the girth \(g\), we use \(\overline{\alpha}_{0,g}\) and \(\overline{\alpha}_{1,g}\). By \autoref{thm:gz3_2}, \autoref{thm:gz3_3} and the result of \cite{JooyandehMOPZ17}, we have \(23\leq\overline{\alpha}_0\leq 40\) and \(27\leq\overline{\alpha}_{0,4}\leq 40\). By \cite[Corollary 11]{GoedgebeurZ19}, we have \(\overline{\alpha}_1\geq 22\) and \(\overline{\alpha}_{1,4}\geq 26\). Wiener \cite{Wiener18} showed that there exists a planar almost hypohamiltonian graph of order 31 and girth 4, which is the smallest known planar almost hypohamiltonian graph (of girth 4), so \(\overline{\alpha}_1\leq\overline{\alpha}_{1,4}\leq 31\). No planar almost hypohamiltonian graph of girth 3 is known \cite{GoedgebeurZ19}.

Thomassen \cite{ThomassenC74} showed that every planar hypohamiltonian
graph contains a cubic vertex. C. T. Zamfirescu \cite{Zamfirescu19} proved that (i) every planar hypohamiltonian graph contains at least four cubic vertices, (ii) every planar almost hypohamiltonian graph contains a cubic vertex which is not the exceptional vertex (solving one of his problems in \cite{Zamfirescu15}), and (iii) every hypohamiltonian graph with crossing number 1 contains a cubic vertex. No planar hypohamiltonian graph with fewer than 30 cubic vertices was known \cite{JooyandehMOPZ17,Zamfirescu19}. Thus, the minimum number of cubic vertices contained in planar hypohamiltonian graphs is between 4 and 30.

The smallest known cubic planar hypohamiltonian graphs have 70 vertices and girth 4. The first was found by Araya and Wiener \cite{AW11}, while six more were found by M. Jooyandeh and B. McKay \cite{Plantri54,MathWorld}. The cubic planar hypohamiltonian graphs must have girth at least 4 \cite{CollierS78,GoedgebeurZ17}. McKay \cite{McKay16} showed that cubic planar hypohamiltonian graphs of girth 5 exist, and that the smallest ones are three graphs on 76 vertices. J. Goedgebeur and C. T. Zamfirescu \cite[Theorem 3.5]{GoedgebeurZ17} proved that the smallest cubic planar hypohamiltonian graph has at least 54 and at most 70 vertices.

The paper is organized as follows. In \autoref{sec:2} we introduce the relations on some known planar hypohamiltonian graphs of order 40 and 42. These relations allow us to construct the smaller ones. In \autoref{sec:3} we present planar hypohamiltonian graphs on 34, 37, 38, and 41 vertices. Hitherto it was unknown whether planar hypohamiltonian graphs of these orders exist. We exploit a special subgraph contained in two graphs of Jooyandeh et al. \cite{JooyandehMOPZ17} to construct two planar hypohamiltonian graphs on 34 vertices. We further modify the 34-vertex graphs to obtain six planar hypohamiltonian graphs on 37 vertices, and rediscover Wiener's planar almost hypohamiltonian graph on 31 vertices \cite{Wiener18}. No planar hypohamiltonian graph with fewer than 30 cubic vertices was known \cite{Zamfirescu19}. The 34-vertex graphs have 26 cubic vertices, narrowing the gap between 4 and 30 \cite{Zamfirescu19}. The smallest known hypohamiltonian graph with crossing number 1 is of order 36 due to Wiener \cite{Wiener18}. We reduce the best-known bound of 36 to 34 by the hypohamiltonian graphs with crossing number 1 obtained by adding an edge to the 34-vertex graphs. By applying an operation introduced by Thomassen \cite{THOMASSEN198136} to the 37-vertex graphs, several planar hypohamiltonian graphs on 41 vertices are obtained, by which we settle the open question whether there exists a planar hypohamiltonian graph on 41 vertices \cite{JooyandehMOPZ17}. In \cite{JooyandehMOPZ17}, all the 25 planar hypohamiltonian graphs on 40 vertices have no nontrivial automorphisms. We present six planar hypohamiltonian graphs on 40 vertices with nontrivial automorphisms.

Our graphs make some progress on Holton's question regarding the smallest natural number \(n_0\) such that there exists a planar hypohamiltonian graph of order \(n\) for every \(n\geq n_0\) \cite[Problem 1.(b)]{Z7Problems}. The best-known bound for \(n_0\) was 42 \cite[Theorem 4.3]{JooyandehMOPZ17}. By showing the existence of planar hypohamiltonian graphs on 41 vertices, we reduce \(n_0\) to 40.

All of our graphs in this paper are constructed by hand. We use computer to check that they are hypohamiltonian and pairwise non-isomorphic. We did not exhaustively try all possible ways for constructing small planar hypohamiltonian graphs. It is possible that other new constructions yield even more graphs.

\section{Equivalence relations and order relations}\label{sec:2}

Planar hypohamiltonian graphs \(G_1,\ldots,G_k\) are said to be \emph{H-equivalent} if each of them can be obtained from a fixed graph \(H\) by adding vertices and edges such that for every \(i,j=1,\ldots,k\), \(|V(G_i)-V(H)|=|V(G_j)-V(H)|\) and \(|E(G_i)-E(H)|=|E(G_j)-E(H)|\), denoted by \(G_i=_H G_j\). Similarly, \(G\) is said to be \emph{H-strictly bigger than} \(G'\) if \(|V(G)-V(H)|>|V(G')-V(H)|\) and \(|E(G)-E(H)|>|E(G')-E(H)|\), denoted by \(G>_H G'\).

\autoref{fig:3ph42}, \autoref{fig:4ph42}, and \autoref{fig:2ph40} are examples of such equivalence relations. These graphs can be found in the House of Graphs \cite{HoG2} by searching the corresponding House-of-Graphs ID (HoG ID), such as the Wiener-Araya Graph \cite{WienerA11} with HoG ID 1431, which we will also call Graph 1431. In the figures, the number on the lower right of each graph is the corresponding HoG ID. In \autoref{sec:3} we will use these relations to construct two planar hypohamiltonian graphs on 34 vertices, as shown in \autoref{fig:2ph34}.

\nextLine

\begin{figure}[H]
\centering
\captionsetup{width=.8\linewidth}
\includegraphics[scale=1.0]{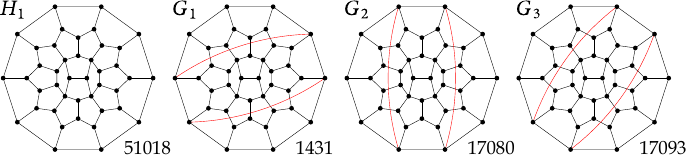}
\caption{\(G_1\), \(G_2\), and \(G_3\) are planar hypohamiltonian graphs of order 42. We have \(G_1=_{H_1} G_2=_{H_1} G_3\).}\label{fig:3ph42}
\end{figure}

\begin{figure}[H]
\centering
\captionsetup{width=.8\linewidth}
\includegraphics{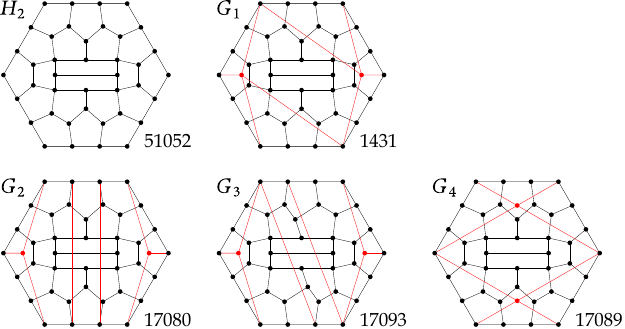}
\caption{\(G_1\), \(G_2\), \(G_3\) and \(G_4\) are planar hypohamiltonian graphs of order 42. We have \(G_1=_{H_2} G_2=_{H_2} G_3=_{H_2} G_4\).}\label{fig:4ph42}
\end{figure}

\begin{figure}[H]
\centering
\captionsetup{width=.8\linewidth}
\includegraphics{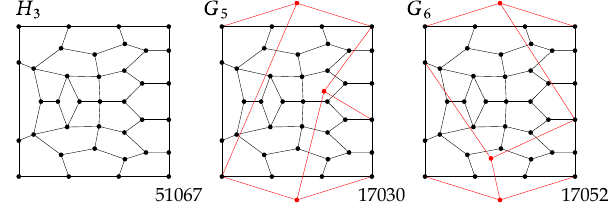}
\caption{\(G_5\) and \(G_6\) are planar hypohamiltonian graphs of order 40 \cite{JooyandehMOPZ17}. We have \(G_5=_{H_3} G_6\).}\label{fig:2ph40}
\end{figure}

\section{Results}\label{sec:3}

To construct planar hypohamiltonian graphs on fewer than 40 vertices, our approach is to modify \(H_3\) to the graph in \autoref{fig:h4}.

\begin{figure}[H]
\centering
\captionsetup{width=.6\linewidth}
\includegraphics[scale=1.4]{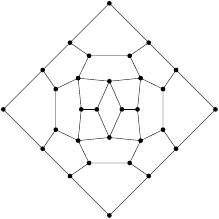}
\caption{Graph 51085, denoted by \(H_4\).}\label{fig:h4}
\end{figure}

\begin{figure}[H]
\centering
\captionsetup{width=.9\linewidth}
\includegraphics[scale=1.4]{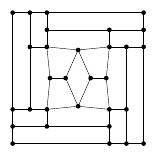}
\caption{An equivalent representation of \(H_4\). This will be used when constructing two planar hypohamiltonian graphs of order 40 with nontrivial automorphisms, Graph 51101 and 51102, as shown in \autoref{fig:6ph40}.}\label{fig:h4e}
\end{figure}

\begin{figure}[H]
\centering
\captionsetup{width=.75\linewidth}
\includegraphics[scale=1.4]{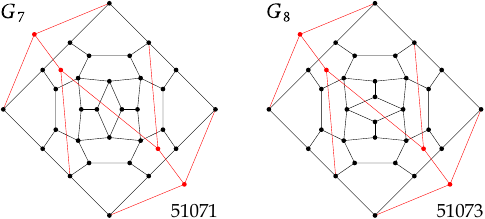}
\caption{Two planar hypohamiltonian graphs of order 34. We have \(G_7=_{H_4} G_8\).}\label{fig:2ph34}
\end{figure}

C. T. Zamfirescu \cite[Problem 1.(a)]{Z7Problems} asked, ``Do planar hypohamiltonian graphs on less than 40 vertices exist?" By the existence of the graphs in \autoref{fig:2ph34}, we settle the problem. No planar hypohamiltonian graph with fewer than 30 cubic vertices was known \cite{JooyandehMOPZ17,Zamfirescu19}. We improve upon the result by the graphs in \autoref{fig:2ph34}, which have 26 cubic vertices.

We tried to add vertices and edges to \(H_4\) by hand in as many ways as possible. Not only can \(H_4\) be used to construct \(G_7\) and \(G_8\), but also two planar hypohamiltonian graphs of order 37 shown in \autoref{fig:6ph37-1}.

\begin{figure}[H]
\centering
\captionsetup{width=.75\linewidth}
\includegraphics[scale=1.4]{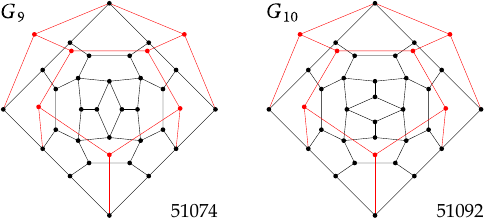}
\caption{Two planar hypohamiltonian graphs of order 37. We have \(G_9=G_{10}>_{H_4} G_7=_{H_4} G_8\). The graphs have 28 cubic vertices.}\label{fig:6ph37-1}
\end{figure}

We rediscover Wiener's planar almost hypohamiltonian graph of order 31 shown in \autoref{fig:w31}, which is the smallest known planar almost hypohamiltonian graph of girth 4. No planar almost hypohamiltonian graph of girth 3 is known \cite{GoedgebeurZ19}.

\begin{figure}[H]
\centering
\captionsetup{width=.9\linewidth}
\includegraphics[scale=1.4]{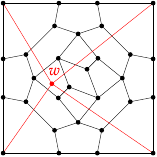}
\caption{Wiener's planar almost hypohamiltonian graph of order 31. The exceptional vertex is the red vertex \(w\). The HoG ID is 51072.}\label{fig:w31}
\end{figure}

We tried by hand to modify \(H_4\), and found four more planar hypohamiltonian graphs on 37 vertices shown in \autoref{fig:6ph37-2} and \autoref{fig:6ph37-3}.

\begin{figure}[H]
\centering
\captionsetup{width=.75\linewidth}
\includegraphics[scale=1.4]{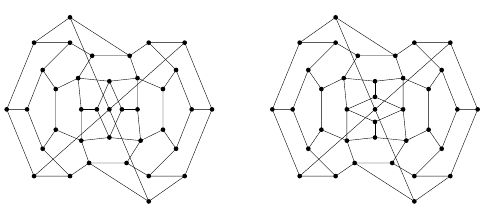}
\caption{(Left) Graph 51093. (Right) Graph 51094.}\label{fig:6ph37-2}
\end{figure}

\begin{figure}[H]
\centering
\captionsetup{width=.7\linewidth}
\includegraphics[scale=1.4]{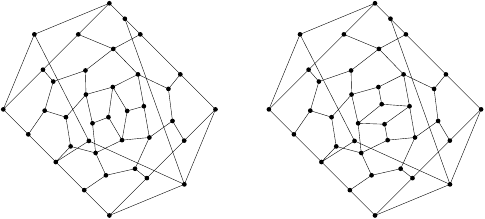}
\caption{(Left) Graph 51095. (Right) Graph 51096. These two graphs have no nontrivial automorphisms.}\label{fig:6ph37-3}
\end{figure}

\begin{theorem}[Grinberg's Theorem {\cite[Theorem 7.3.5]{West2nd}}]
Let \(G\) be a simple plane graph. An \(i\)-face is a face of \(G\) which has size \(i\). If \(G\) has a Hamiltonian cycle \(C\) and \(f_i\) (\(f'_i\)) is the number of \(i\)-faces of \(G\) inside (outside) of \(C\), we have \[\sum_{i\geq 3}(i-2)(f_i-f'_i)=0.\]
\end{theorem}

\begin{theorem}\label{thm:phg}
The graphs shown in \autoref{fig:2ph34}, \autoref{fig:6ph37-1}, \autoref{fig:6ph37-2}, and \autoref{fig:6ph37-3} are planar and hypohamiltonian.
\end{theorem}

\begin{proof}
It is not hard to check that these graphs are planar. We follow the proof of \cite[Theorem 4.1]{JooyandehMOPZ17} to show that \(G_7\) is non-Hamiltonian. For the other graphs, the proof is the same. 

Assume to the contrary that \(G_7\) contains a Hamiltonian cycle \(C\), which must then satisfy Grinberg's Theorem. The graph \(G_7\) contains five 4-faces and 18 5-faces. Then \[\mmod{\sum_{i\geq 3}(i-2)(f_i-f'_i)\equiv 2(f_4-f'_4)}{0}{3},\] where \(f_4+f'_4=5\). So \(f'_4=1\) and \(f_4=4\), or \(f'_4=4\) and \(f_4=1\). Let \(Q\) be the 4-face on a different side from the other four.

Notice that an edge belongs to a Hamiltonian cycle if and only if the two faces it belongs to are on different sides of the cycle. Let the 4-face at the center of \(H_4\) be \(q_c\). Since \(q_c\) has edges in common with all other 4-faces and its edges cannot all be in a Hamiltonian cycle, \(q_c\) cannot be \(Q\).

If \(Q\) is any of the other 4-faces, denote the edge belonging to \(Q\) and \(q_c\) by \(e=\{a,b\}\) with \(\deg(a)=4\) and \(\deg(b)=3\).
Denote the other two vertices of \(q_c\) by \(c\) and \(d\) with \(\deg(c)=4\) and \(\deg(d)=3\). Since the vertex \(a\) belongs to the edge \(e\), the edge \(\{d,a\}\) is not a part of \(C\). The edge \(\{d,c\}\) is not a part of \(C\), either. It is because the two 4-faces that \(\{d,c\}\) belongs to are on the same side of \(C\). Thus, \(G_7\) is non-Hamiltonian.

Finally, for each vertex it is routine to exhibit a cycle of length 33 (\autoref{fig:2ph34}) or 36 (\autoref{fig:6ph37-1}, \autoref{fig:6ph37-2}, and \autoref{fig:6ph37-3}) that
avoids it.
\end{proof}

The smallest known hypohamiltonian graph with crossing number 1 due to Wiener \cite{Wiener18} has 36 vertices. We reduce the bound of 36 to 34 by the hypohamiltonian graphs with crossing number 1 obtained by adding an edge to the graphs in \autoref{fig:2ph34}. We present six such graphs shown in \autoref{fig:6h34c1} derived from \(G_7\) in \autoref{fig:2ph34} by adding an edge (shown in red) to it.

\begin{figure}[H]
\centering
\captionsetup{width=.8\linewidth}
\includegraphics[scale=1]{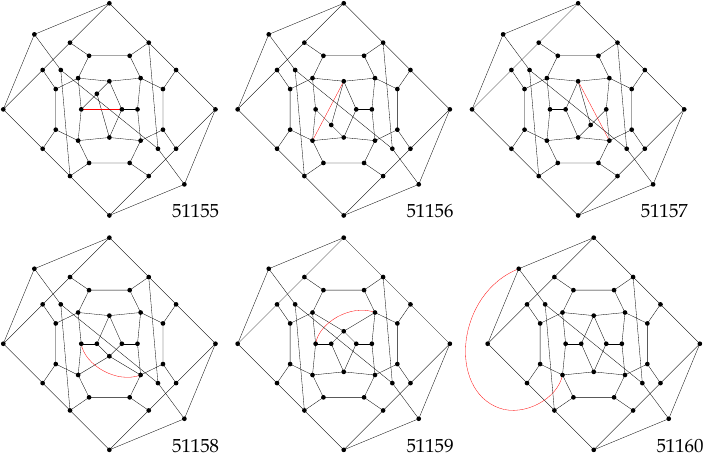}
\caption{Six hypohamiltonian graphs with crossing number 1 derived from \(G_7\) in \autoref{fig:2ph34}.}\label{fig:6h34c1}
\end{figure}

We use computer to check that the graphs in \autoref{fig:6h34c1} are hypohamiltonian and pairwise non-isomorphic. It is not hard to check by hand that each red edge must cross over at least one black edge, so the crossing number is 1.

Whether there exists a planar hypohamiltonian graph on 41 vertices was
an open question \cite{JooyandehMOPZ17}. The Holton's question is what is the smallest natural number \(n_0\) such that there exists a planar hypohamiltonian graph of order \(n\) for every \(n\geq n_0\) \cite[Problem 1.(b)]{Z7Problems}. Motivated by Holton's question, we now need an operation introduced by Thomassen \cite{THOMASSEN198136} with which he showed that there exist infinitely many planar cubic hypohamiltonian graphs. Let \(G\) be a graph containing a 4-cycle \(v_1v_2v_3v_4=C\). We denote by \(\mathrm{Th}(G,C)\) the graph obtained from \(G\) by deleting the edges \(v_1v_2\), \(v_3v_4\) and adding a new 4-cycle \(v'_1v'_2v'_3v'_4\) and the edges \(v_iv'_i\), \(1\leq i\leq 4\). Araya and Wiener \cite{WienerA11} note that a result in \cite{THOMASSEN198136} generalizes as follows, with the same proof.

\begin{lemma}[{\cite[Lemma 4.4]{WienerA11}}]\label{lmm:aw} Let \(G\) be a planar hypohamiltonian graph having a \(4\)-cycle \(v_1v_2v_3v_4=C\) and suppose that the degrees of the vertices \(v_1,v_2,v_3,v_4\) are \(3\). Then \(Th(G,C)\) is also a planar hypohamiltonian graph.
\end{lemma}

The assertion of \hyperref[lmm:aw]{Lemma 3.3} holds if one uses the condition that at least one of \(v_1,v_2,v_3,v_4\) is cubic in \(G\). It also holds if the edges \(v_1v_2\) or \(v_3v_4\) (prossibly both) lie is \(G\) \cite[Theorem 5]{Zamfirescu19}.

Jooyandeh et al. \cite[Theorem 4.3]{JooyandehMOPZ17} used the operation Th to show that planar hypohamiltonian graphs of order \(n\) for every \(n\geq 42\), which follows \(n_0\leq 42\). By applying the operation Th to the 37-vertex graphs in \autoref{fig:6ph37-1}, \autoref{fig:6ph37-2}, and \autoref{fig:6ph37-3}, we obtain a number of planar hypohamiltonian graphs on 41 vertices, and thus settle the open question of Jooyandeh et al. \cite{JooyandehMOPZ17}. \autoref{fig:2ph41} shows two such 41-vertex graphs obtained by applying the operation Th to the 4-cycle at the center of \(H_4\) without deleting the edges \(v_1v_2\) and \(v_3v_4\). The other graphs can be found in the House of Graphs \cite{HoG2} by searching for
the text string ``planar hypohamiltonian graph".

\begin{figure}[H]
\centering
\captionsetup{width=.9\linewidth}
\includegraphics[scale=1]{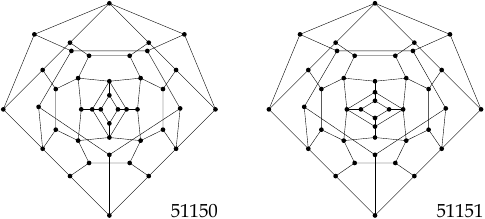}
\caption{Two planar hypohamiltonian graphs on 41 vertices, obtained by applying the operation Th to the graphs in \autoref{fig:6ph37-1}.}\label{fig:2ph41}
\end{figure}

By the graphs in \autoref{fig:2ph41}, the result of Jooyandeh et al. \cite[Theorem 4.3]{JooyandehMOPZ17} is improved further in the next theorem.

\begin{theorem}\label{thm:all40}
There exist planar hypohamiltonian graphs of order \(n\) for every \(n\geq 40\).
\end{theorem}

By applying the operation Th to the graphs in \autoref{fig:2ph34}, we have several planar hypohamiltonian graphs on 38 vertices, two of which are shown in \autoref{fig:2ph38}. Whether there exists planar hypohamiltonian graphs on 35, 36, 39 vertices remains
an open question.

\begin{figure}[H]
\centering
\captionsetup{width=.9\linewidth}
\includegraphics[scale=1]{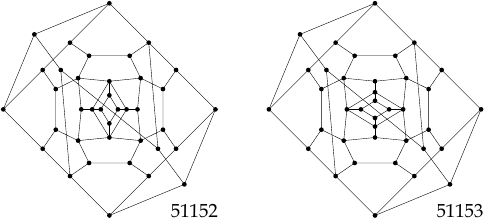}
\caption{Two planar hypohamiltonian graphs on 38 vertices, obtained by applying the operation Th to the graphs in \autoref{fig:2ph34}.}\label{fig:2ph38}
\end{figure}

The 25 planar hypohamiltonian graphs on 40 vertices in \cite{JooyandehMOPZ17} have no nontrivial automorphisms, as mentioned in \cite[Conclusions]{JooyandehMOPZ17}. In this paper, we present six planar hypohamiltonian graphs on 40 vertices with nontrivial automorphisms shown in \autoref{fig:6ph40}. Additionally, the planar hypohamiltonian graphs in \autoref{fig:2ph34}, \autoref{fig:6ph37-1}, \autoref{fig:6ph37-2}, \autoref{fig:2ph41}, and \autoref{fig:2ph38} have nontrivial automorphisms with the corresponding orders 34, 37, 41, and 38. In the \hyperref[appdx]{Appendix}, we also list some planar hypohamiltonian graphs on 43, 46, and 49 vertices with nontrivial automorphisms.

\begin{figure}
\centering
\captionsetup{width=.9\linewidth}
\includegraphics[scale=1]{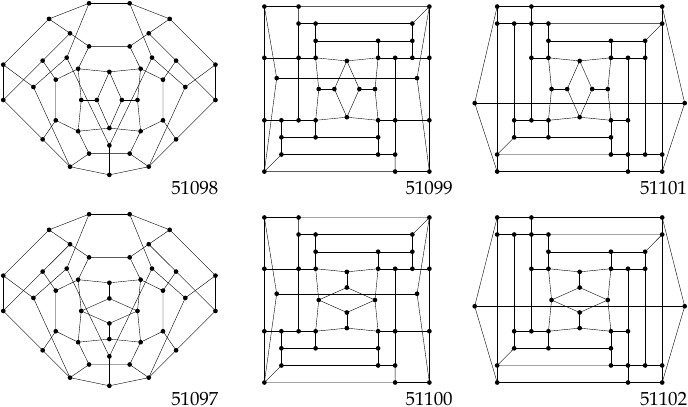}
\caption{Six planar hypohamiltonian graphs on 40 vertices with nontrivial automorphisms. The two graphs on the rightmost contain \(H_4\).}\label{fig:6ph40}
\end{figure}

In 2011, Araya and Wiener \cite[Corollary 3.1; Theorem 4.9]{WienerA11} proved that there exist planar hypotraceable graphs on \(162+4k\) vertices for every \(k\geq 0\), and on \(n\) vertices for every \(n\geq 180\). In 2017, Jooyandeh et al. \cite[Theorem 4.5]{JooyandehMOPZ17} proved that there exist planar hypotraceable graphs on 154 vertices, and on \(n\) vertices for every \(n\geq 156\). In 2018, Wiener \cite{Wiener18} constructed a planar hypotraceable graph of order 138, improving the bound of 154 in \cite{JooyandehMOPZ17}. To improve the best-known bound of 138, we make use of the following theorem, which is a slight modification of \cite[Lemma 3.1]{Thomassen74a}.

\begin{theorem}[{\cite[Theorem 4.4]{JooyandehMOPZ17}}]\label{thm:j4_4}
Let \(G_i=(V_i,E_i),\ 1\leq i\leq 4,\) be four planar hypohamiltonian
graphs. Then there is a planar hypotraceable graph of order \(|V_1|+|V_2|+|V_3|+|V_4|-6\).
\end{theorem}

Combining \autoref{thm:j4_4} with \autoref{thm:phg} and \autoref{thm:all40}, we obtain the following.

\begin{theorem}\label{thm:pht130}
There exist planar hypotraceable graphs on \(130\) vertices, and on \(n\) vertices for every \(n\geq 154\).
\end{theorem}

Let \(H\) be a cubic graph and \(G\) be a graph containing a cubic vertex \(w\in V(G)\). We say that we \emph{insert} \(G\) into \(H\), if we replace every vertex of \(H\) with \(G-w\) and connect the endpoints of edges in \(H\) to the neighbours of \(w\). We improve upon the previous results \cite{Zamfirescu15,JooyandehMOPZ17,GoedgebeurZ17} as follows. Our arguments follow the same strategy as used in \cite[Corollary 4.7]{JooyandehMOPZ17} and \cite[Theorem 4.1]{ZamfirescuPhD}.

\begin{corollary}
\[27\leq h_4\leq 34,\quad \overline{\alpha}_0\leq 34,\quad \overline{C^1_3}\leq 34,\quad \overline{C^2_3}\leq 2205,\quad \overline{P^1_3}\leq 132,\quad \overline{P^2_3}\leq 8694.\]
\end{corollary}

\begin{proof}
The first inequality is obtained from \cite[Theorem 3.3]{GoedgebeurZ17} and the graphs in \autoref{fig:2ph34}, which improves upon the third inequality in \cite[Corollary 3.4]{GoedgebeurZ17}.

The second and third inequalities follow immediately from the graphs in \autoref{fig:2ph34}. To obtain the last three inequalities, let \(G\) be one of the graphs in \autoref{fig:2ph34}.

For the fourth inequality, insert \(G\) into the 70-vertex planar cubic hypohamiltonian graph\textemdash which we will call here \(G_0\)\textemdash constructed by Araya and Wiener in \cite{AW11}. (\(G_0\) is the smallest known planar cubic hypohamiltonian graph.) This means that each vertex of \(G_0\) is replaced by \(G\) minus some vertex of degree 3. We denote the resulting graph by \(G'\). Araya and Wiener proved \cite{AW11} (using a computer) that every pair of edges in \(G_0\) is missed by a longest cycle. Combining this fact with the hypohamiltonicity of \(G\) and \(G_0\), we obtain that in \(G'\)
any pair of vertices is avoided by a longest cycle. This property is not lost if all edges originally belonging to \(G_0\) are contracted. By construction, the order of \(G'\) is \((34-1)\cdot 70=2310\). Since \(|E(G_0)|=105\), after contracting all edges originally belonging to \(G_0\), we obtain \(\overline{C^2_3}\leq 2310-105=2205\).

In order to prove the fifth inequality, insert \(G\) into \(K_4\). We obtain a graph in which every vertex is avoided by a path of maximal length. Thus, we have \(\overline{P^1_3}\leq 4\cdot(34-1)=132\). The upper bound for \(\overline{P^1_3}\) is not improved upon by \autoref{thm:pht130}. The planar hypotraceable graphs on 130 vertices are constructed by \cite[Lemma 3.1; Fig. 2]{Thomassen74a}, which are 2-connected.

For the last inequality, consider the graph \(G_0\) from the third paragraph of this proof and insert \(G_0\) into \(K_4\) to obtain \(H\). Now insert \(G\) into \(H\). Finally, contract all edges which originally belonged to \(H\).
\end{proof}

For the cubic planar case, the smallest known cubic planar hypohamiltonian graphs have 70 vertices and girth 4. The first one was found by Araya and Wiener \cite{AW11}, and the other six were found by Jooyandeh and McKay \cite{Plantri54,MathWorld}. The cubic planar hypohamiltonian graphs must have girth at least 4 \cite{CollierS78,GoedgebeurZ17}. McKay \cite{McKay16} showed that cubic planar hypohamiltonian graphs of girth 5 exist, and that the smallest ones are three graphs on 76 vertices. Therefore the smallest cubic planar hypohamiltonian graph must have girth exactly 4. J. Goedgebeur and C. T. Zamfirescu \cite[Theorem 3.5]{GoedgebeurZ17} proved that the smallest cubic planar hypohamiltonian graph has at least 54 and at most 70 vertices.

\section{Conclusions}\label{sec:4}

Based on the work of \cite{JooyandehMOPZ17}, the two planar hypohamiltonian graphs of order 34 with nontrivial automorphisms in this paper significantly decrease the gap between the lower and upper bounds for the order of the smallest planar hypohamiltonian graph. By the further study of the 34-vertex graphs, we obtain six planar hypohamiltonian graphs of order 37, four of which have nontrivial automorphisms. By applying the operation introduced by Thomassen \cite{THOMASSEN198136} to the 34- and 37-vertex graphs, we obtain several planar hypohamiltonian graphs of order 38 and 41, some of which have nontrivial automorphisms.

By showing the existence of planar hypohamiltonian graphs on 41 vertices, we settle the open question in \cite{JooyandehMOPZ17} whether there exist planar hypohamiltonian graphs on 41 vertices. Our 41-vertex graphs make some progress on Holton's question regarding the smallest
natural number \(n_0\) such that there exists a planar hypohamiltonian graph of order \(n\) for every \(n\geq n_0\) \cite[Problem 1.(b)]{Z7Problems}. Our graphs gives \(n_0\leq 40\), improving upon the result of \cite{JooyandehMOPZ17}. 

Quote the authors of \cite{JooyandehMOPZ17}, ``An exhaustive study of graphs with prescribed automorphisms might lead to the discovery of new, smaller graphs." Their point of view was precise. Throughout the graphs in \autoref{fig:2ph34}, \autoref{fig:6ph37-1}, and \autoref{fig:6ph37-2}, each of them was constructed with some prescribed symmetries. Finally, new and smaller graphs were found. The further study of these graphs leads to six planar hypohamiltonian graphs of order 40 with nontrivial automorphisms.

Whether there exists planar hypohamiltonian graphs on 35, 36, and 39 vertices remains an open question.

Our graphs are based on two graphs of \cite{JooyandehMOPZ17}, Graph 17030 and 17052, shown in \autoref{fig:2ph40}. In future work one could try the other 23 graphs of \cite{JooyandehMOPZ17} to see whether there are new planar hypohamiltonian graphs on fewer than 40 vertices. Additionally, all the graphs in this paper were constructed by hand by adding vertices and edges to the graph in \autoref{fig:h4}, or modifying it. We did not try all possible for constructing small planar hypohamiltonian graphs. An
exhaustive study of our construction might lead to the discovery of new, smaller graphs.

We can hope that the current work inspires further progress in finding the smallest (cubic) planar hypohamiltonian or smallest planar almost hypohamiltonian graphs.

\section*{Acknowledgements}

We thank Jan Goedgebeur and Carol T. Zamfirescu for their helpful comments. We thank Brendan McKay and Eric W. Weisstein for their comments on cubic planar hypohamiltonian graphs, and thank Eric W. Weisstein for the comment on Figure 11.

\section*{Appendix}\label{appdx}
\begin{figure}[H]
\centering
\captionsetup{width=.9\linewidth}
\includegraphics[scale=1]{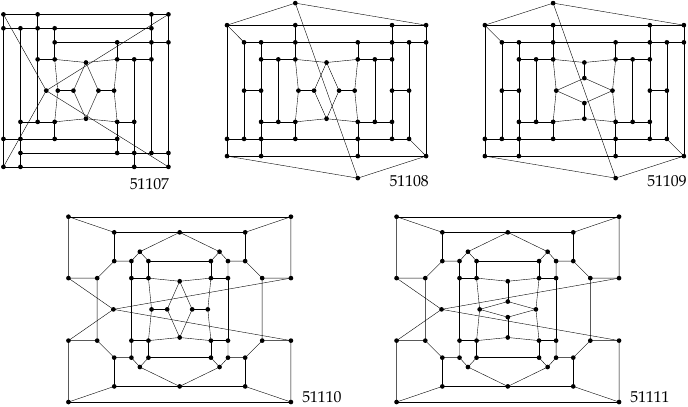}
\caption{Planar hypohamiltonian graphs on 43, 46, and 49 vertices. Note that Graph 51107 contains \(H_4\).}\label{fig:6ph43more}
\end{figure}

\bibliography{bib}

\end{document}